 \documentclass[twoside,12pt,reqno]{amsart}
\usepackage{soul,cite,lineno,cancel}
\usepackage{booktabs,multirow,slashbox,xcolor,verbatim,relsize,tikz,tikz-cd,rotating,lscape,soul}
\usepackage{hyperref,enumitem}
\usepackage{latexsym,amsmath,amsthm,amsfonts,amssymb,mathtools}

\newtheorem{theorem}{Theorem}[section]
\newtheorem{definition}[theorem]{Definition}

\newtheorem{remark}[theorem]{Remark}
\newtheorem{proposition}[theorem]{Proposition}
\newtheorem{corollary}[theorem]{Corollary}

\newcommand{\ve}{\varepsilon}

\newcommand{\vO}{\varOmega}

\def\N{{\mathbb N}}

\def\be{\begin{equation}}
 \newcommand{\dint}{\displaystyle{\int}}
\def\ee{\end{equation}}
\def\ba*{\begin{eqnarray*}}
	\def\ea*{\end{eqnarray*}}
\def\ba{\begin{eqnarray}}
\def\ea{\end{eqnarray}}
\def\smi{\setminus}

\begin{document}

\title[]{Convergence for varying measures in the  topological case}
\author{L. Di Piazza, V. Marraffa,
  K. Musia{\l}, A. R. Sambucini}
\begin{abstract}
In this paper convergence  theorems  for  sequences of scalar, vector and multivalued Pettis integrable
functions on a topological measure space
are proved for varying measures vaguely convergent.
\end{abstract}
\newcommand{\Addresses}{{
  \bigskip
  \footnotesize
\textit{Luisa Di Piazza and Valeria Marraffa}:
 Department of Mathematics, University of Palermo, Via Archirafi 34, 90123 Palermo, (Italy).
 Emails: luisa.dipiazza@unipa.it, valeria.marraffa@\-unipa.it, Orcid ID: 0000-0002-9283-5157,
0000-0003-1439-5501
\\
\textit{Kazimierz Musia{\l}}:
  Institut of Mathematics, Wroc{\l}aw University, Pl. Grunwaldzki  2/4, 50-384 Wroc{\l}aw, (Poland).
  Email: kazimierz.musial@math.uni.wroc.pl, Orcid ID: 0000-0002-6443-2043 \\
\textit{Anna Rita Sambucini
}:
 Department of Mathematics and Computer Sciences, 06123 Perugia, (Italy).
 Email: anna.sambucini@unipg.it, Orcid ID: 0000-0003-0161-8729; ResearcherID: B-6116-2015.
}}
\date{\today}

\maketitle

\section{Introduction}\label{s-intro}
Conditions for the  convergence of sequences of measures $(m_n)_n$ and of 
 their integrals $(\int f_n dm_n)_n$  in
a measurable space $\Omega $ are of  interest in many areas  of pure and applied mathematics
such as  statistics,    transportation problems, interactive partial systems,
neural networks and  signal  processing
(see, for instance, \cite{MIMMO,carlo,avendano,danilo,xdanilo,mdpi,gal,Cox2}).
In particular, for the image reconstruction, which is a branch of signal theory,  in the last years,
interval-valued functions have been considered since the process of discretization of an image  is affected by
 quantization errors (\cite{mesiar}) and  its  numerical  approximation
can be interpreted as  a suitable sequence of interval-valued functions
 (see for instance  \cite{latorre,Pap-}).

Obviously,   suitable convergence notions  are needed for the varying measures, see for example
\cite{Fein-arX1807,Fein-arX1902,lerma,lasserre,serfozo,Ma,M-S}  and the references therein.
 In a previous paper \cite{dmms} we have  examined the problem
 when the varying measures converge setwisely   in an arbitrary measurable space. This type of convergence is a powerful tool
 since it permits to obtain strong results,  for example the Vitali-Hahn-Saks Theorem or a Dominated Convergence Theorem  \cite{lerma}.\\
 But sometime in the  applications it is difficult, at least technically,
  to prove that the sequence  $(m_n(A))_n$  converges to $m(A)$ for every  measurable set $A$, unless e.g.
the sequence  $(m_n)_n$ is decreasing or increasing.
 So other types of convergence are studied, based on  the structure of the topological space $\Omega$,
  such as the vague and the weak convergence which are,  in general, weaker than the setwise.
These convergences  are useful, for example,  from the point of view of
applications on non-interactive particle systems (see \cite{
Cox2, Ma}). 

In the present  paper we continue the research started in  \cite{dmms}
and we  provide sufficient conditions in order to  obtain    Vitali's type  convergence results for a sequence of
(multi)functions $(f_n)_n$  integrable  with respect to a sequence  $(m_n)_n$ of measures when  $(m_n)_n$ converges
 vaguely or weakly to a finite measure $m$.\\
The known results, in literature, as far as we know, require that the topological space $\Omega$,
endowed with the Borel $\sigma$-algebra is a  metric  space (\cite{Fein-arX1807,Fein-arX1902}),
or a locally compact  space which is also:  separable and metric  ( \cite{lerma}),
metrizable (\cite{Kall}) or Hausdorff  second countable (\cite{serfozo}). An interesting comparison among all these results is given in \cite{Ma}.

In the present paper, following the ideas of Bogachev (\cite{Bog}),   we assume that $\Omega$ is only  an arbitrary  locally compact  Hausdorff space.
The paper is organized as follows: in Section \ref{sec2} the topological structure of the space $\Omega$ is introduced
 together with the convergence types considered and some of their properties.
In Section \ref{sec3} the scalar case is studied;
the main result of this section is Theorem \ref{Th1-4},  where we obtain the convergence of the integrals
 $(\int f_n dm_n)_n$ over arbitrary Borel sets under suitable conditions.
In Section  \ref{sec4}  Theorem \ref{Th1-4}  is  applied in order to obtain analogous results for the multivalued case,
obtaining as a corollary also the vector case. In both cases the Pettis integrability of the integrands is considered.
Finally, adding a condition as in \cite[Theorem 3.2]{dmms} we obtain a convergence result for (multi)functions
in Proposition \ref{Thmulti2bis} on measurable spaces.

\section{Topological case, preliminaries}\label{sec2}
Let  $\Omega$ be a locally compact
Hausdorff space and  let $ \mathcal{B}$ be its
Borel  $\sigma$-algebra.
The symbol $\mathcal{F}(\Omega)$  indicates the class of all
$\mathcal{B}$-measurable functions  $f:\Omega \to \mathbb{R}$.
 We denote by
$C(\Omega)$, $C_0(\Omega)$, $C_c(\Omega)$ and $C_b(\Omega)$ respectively the family of all continuous functions,
 and the subfamilies of all continuous functions  that vanish at infinity, have compact support,
  are  bounded.\\
  
Throughout, we will use  Urysohn's Lemma in the form  (\cite[Lemma 2.12]{rudin}): 
\begin{itemize}
	\item  If $K$ is compact and $U\supset{K}$ is open in a locally compact $\Omega$, 
	then there exists  $f:\Omega\to[0,1]$,  $f\in C_c(\Omega)$, such that $\chi_K\leq{f}\leq  \chi_U$.
\end{itemize}
 All the measures we will consider on $(\Omega, \mathcal{B})$ are finite and by  $\mathcal{M}(\Omega)$ we denote the family of {\it finite nonnegative measures}.  
As usual  a  measure $m \in \mathcal{M}(\Omega)$ is Radon if
it is inner regular in the sense of approximation by compact sets.\\

We recall the following definitions of convergence for  measures.\\

\begin{definition}\label{v-w} \rm
\rm Let $m$ and  $m_n$ be in $\mathcal{M}(\Omega)$. We say that
\begin{description}
\item[\ref{v-w}.a)]
$(m_n)_n$ {\it converges vaguely  } to $m$ ($m_n \xrightarrow[]{v}m$)  (\cite[Section 2.3]{lerma})
 if
\[ \int_{\Omega} g dm_n \to \int_{\Omega}  g dm, \quad \mbox{for every  }  g  \in C_0(\Omega).\]
\item[\ref{v-w}.b)]
$(m_n)_n$   {\it  converges weakly}  to $m$  ($m_n \xrightarrow[]{w} m$) (\cite[Section 2.1]{lerma})
 if
\[ \int_{\Omega} g dm_n \to \int_{\Omega}  g dm, \quad \mbox{for every  }   g  \in C_b(\Omega).\]
\item[\ref{v-w}.c)]
 {\it   $(m_n)_n$ {\it converges   setwisely} to $m$ ($m_n \xrightarrow[]{s}  m$)	 if \,
$\lim_n m_n (A) = m(A)$ for every $A \in \mathcal{B}$  {\rm (\hskip-.03cm \cite[Section 2.1]{lerma},
\cite[Definition 2.3]{Fein-arX1902})}} or,
 equivalently (\cite{Ma}), if
\[ \int_{\Omega} g dm_n \to \int_{\Omega}  g dm, \quad \mbox{for every bounded }   g \in \mathcal{F}(\Omega).\]
\item[\ref{v-w}.d)]
$(m_n)_n$ is {\it uniformly absolutely continuous with respect to $m$} if for each
$\varepsilon > 0$ there exists $\delta > 0$ such that
\begin{equation}
\left(E \in  \mathcal{B} \quad \mbox{and} \ \  m(E)<\delta \right)  \,\, \Longrightarrow \,\, \sup_n  m_n (E) < \varepsilon. 
\label{auc2}
\end{equation}
We would like to note that the condition $m_n \leq  m$, for every $n \in\N$, implies  that
$(m_n)_n$ is uniformly absolutely continuous with respect to $m$.
\end{description}
\end{definition}

\begin{remark}\label{RWV}
\rm As observed in \cite{lerma} the setwise convergence is stronger than  the vague and  the weak convergence.
For the converse implications we know, by \cite[Lemma 4.1 (ii)]{lasserre}, that
if $(m_n)_n$ is a sequence in   $\mathcal{M}(\Omega)$  with $m_n \leq  m$, where $m \in  \mathcal{M}(\Omega)$  and
 $(m_n)_n$  converges vaguely to $m$, then $(m_n)_n$  converges  setwisely  to $m$.
If $m$ is  $\overline{\mathbb{R}}$-valued this is not true in general, see for example  \cite[page 143]{lerma}.
The weak convergence is stronger than the vague convergence; as an example we can consider $m_n:=\delta_n $
 (the Dirac measure at the point $x=n$) and $m:=0$. The sequence $(m_n)_n$ converges vaguely to $m$,
 but since $m_n(\mathbb{R})=1 \not\to 0 = m(\mathbb{R})$ the convergence cannot be weak.\\\
Moreover we note that if $(m_n)_n$  converges weakly to $m$, then  $m_n(\Omega) \to m(\Omega)$ (it is enough to take $g=1$ in the definition).\\
 \end{remark}
We  have
\begin{proposition}\label{pw}
Let  $m_n$, $n \in \mathbb{N}$, and  $ m$ be  in $\mathcal{M}(\Omega)$, with $m$ Radon.
If $(m_n)_n $  is uniformly absolutely continuous with respect to $m$ and $(m_n)_n$ is vaguely  convergent to $m$, then
$(m_n)_n$ is weakly convergent to $m$.

\end{proposition}
\begin{proof}
We fix $\ve>0$ and let  $f\in{C_b(\Omega)}$. We set 
\(c := \max \{1,\sup_{\Omega} \vert f(\omega)\vert  \} \), let
 $\delta \in ]0, \varepsilon[$ be taken in such a way that  if $E$ is a Borel set with  $m(E)<\delta$, then
\[
\max \left\{\int_E\vert f\vert \,dm, \, \sup_n m_n(E) \right\}<\ve.\]
Let $K$ be a compact set such that $m(K^c)<\delta$. Then by   Urysohn's Lemma  let $h:\Omega\to [0,1]$
be a continuous function  with compact support such that $h(\omega)=1$ for $\omega\in{K}$. Let $g:=f\cdot{h}$.
 Then $g\in{C_0(\Omega)}$. We have for sufficiently large $n\in\N$, depending on the vaguely convergence,
\begin{align}
 & \Big\vert  \int_{\Omega} f dm - \int_{\Omega} f dm_n  \Big\vert  \nonumber \\
	 &\leq \int_{\Omega} \vert f - g\vert  dm +\int_{\Omega} \vert f - g\vert  dm_n + 
	\Big\vert  \int_{\Omega} g \ d m - \int_{\Omega} g \ dm_n \Big\vert  \nonumber \\
 &=  \int_{K^c} \vert f\vert  \cdot  \vert 1-h\vert \,dm+\int_{K^c} \vert f\vert \cdot \vert 1-h\vert \,dm_n+
 \Big\vert  \int_{\Omega} g \ d m - \int_{\Omega} g \ dm_n \Big\vert 
\nonumber\\
 &\leq
\varepsilon (c +2). \nonumber
 \end{align}
 \end{proof}
For other relations among weak or vague convergence and setwise convergence see \cite[Lemma 4.1]{lasserre}. Moreover

\begin{proposition}\label{L4}
Let  $m_n$, $n \in \mathbb{N}$, and  $m$ be in   $\mathcal{M}(\Omega)$ with $m$ Radon.
If $m_n \leq m$, for every $n \in\N$,   and $(m_n)_n$ is vaguely  convergent to $m$, then for every
$f\in{L^1(m)}$ and  $A \in \mathcal{B}$
\begin{equation}
 \lim_n \int_{A} f\,dm_n =  \int_{A} f\,dm.
\label{Prop setwise}
  \end{equation}
In particular $(m_n)_n$ converges to $m$ setwisely.
\end{proposition}

\begin{proof}
 Let $f\in{L^1(m)}$ be fixed. Given  $\varepsilon > 0$ there exists  $g \in C_c(\Omega)$
 such that
\begin{equation}
\int_{\Omega}\vert f-g\vert  dm_n \leq \int_{\Omega}\vert f-g\vert  dm <
\frac{ \varepsilon}{3}.
\label{e1}
\end{equation}
Moreover, since $(m_n)_n$ is vaguely  convergent to $m$, let
$N(\varepsilon/3) $ be such that
 \begin{equation}
\Big\vert  \int_{\Omega} g \ dm - \int_{\Omega} g \ dm_n \Big\vert   < \frac{\varepsilon}{3}
\label{e2}
 \end{equation}
 for $n>N$.
 Therefore by \eqref{e1}  and \eqref{e2} for $n>N$ we obtain
 \begin{align}
& \left\vert  \int_{\Omega} f dm - \int_{\Omega} f dm_n  \right\vert  \nonumber \\
&\leq \int_{\Omega} \vert f - g\vert dm +\int_{\Omega} \vert f - g\vert dm_n +
\left\vert  \int_{\Omega} g \ dm- \int_{\Omega} g \ dm_n \right\vert   < \varepsilon. \nonumber
 \end{align}
Now if $A \in \mathcal{B}$, also $f\chi_A \in L^1(m)$ and (\ref{Prop setwise}) follows. In particular  \mbox{$m_n \xrightarrow[]{s}  m$.}
\end{proof}

Results of the previous type are contained for example in \cite[Proposition 2.3]{lerma} for the setwise convergence when  the  measures $m_n$  are equibounded by  a measure $\nu$ for non negative   $f \in L^1(\nu)$    or in
 \cite[Proposition 2.4]{lerma} under the additional hypothesis of separability of $\Omega$ for non negative and
lower semicontinuous functions $f$. \\

We now introduce the following definition

\begin{definition}\label{a-c}{\rm}
	\rm Let  $(m_n)_n$ be  a sequence in $\mathcal{M}(\Omega)$. We say that:
\begin{description}	
\item[\ref{a-c}.a)]
	A  sequence  $(f_n)_n \subset \mathcal{F}(\Omega)$
	has {\it  uniformly absolutely   continuous  $(m_n)$-integrals   on $\varOmega$},
	if for every $ \varepsilon  >0$ there exists $\delta >0$ such that 	
	 for every $n \in \mathbb{N}$			
	\begin{equation}
			\left(A \in  \mathcal{B} \quad \mbox{and} \ \
		 m_n(A)<\delta \right)  \ \ \Longrightarrow \ \ \int_A\vert f_n\vert \,dm_n < \varepsilon.
\label{Pettis1}
	\end{equation}
	Analogously a function  $f \in  \mathcal{F}(\Omega)$
	has {\it  uniformly absolutely   continuous  $(m_n)$-integrals   on $\varOmega$} if previous condition (\ref{Pettis1})
	 holds for  $f_n:=f$ for every $n \in \mathbb{N}$.		
\item[\ref{a-c}.b)] A  sequence $(f_n)_n \subset \mathcal{F}(\Omega)$
	is  {\it  uniformly $(m_n)$-integrable on $\varOmega$}  if
	\begin{equation}
		\lim_{\alpha  \to +\infty } \sup_n \int_{\vert f_n\vert  >\alpha} \vert f_n\vert \, dm_n \  =0.
	\label{unintegr}
	\end{equation}
\end{description}		
\end{definition}

\begin{remark}\label{2.6jmaa}
\rm
 As we observed in \cite[Proposition 2.6]{dmms} if $(m_n)_n$ is  a bounded sequence of  measures
and 	 $(f_n)_n \subset \mathcal{F}$, then,  $(f_n)_n$
 is   uniformly $(m_n)$-integrable on $\vO$ if and only if it has uniformly absolutely continuous
  $(m_n)$-integrals and
\begin{equation}
	\sup_n \int_{\Omega} \vert f_n\vert \, dm_n < +\infty\,.
		\label{e12}
\end{equation}
\end{remark}

\section{The scalar case}\label{sec3}

\begin{proposition}\label{p1}
Let  $(m_n)_n$  be a sequence in $\mathcal{M}(\Omega) $   which is uniformly absolutely continuous with
 respect to  a Radon measure  $m \in \mathcal{M}(\Omega)$  and vaguely  convergent to  $m$.
 Let  $f \in C(\Omega)$ be a function  which
 has uniformly absolutely continuous $(m_n)$-integrals on $\vO$. Then,
\begin{equation}
	\sup_n \int_{\Omega} \vert f \vert  \, dm_n < +\infty\,.
	\label{e6}
\end{equation}
\end{proposition}
\begin{proof}
 Let  $\varepsilon>0$ be fixed and  let $\sigma=\sigma(\varepsilon)$ be that of the uniform absolutely continuous $(m_n)$-integrability of $f$ as
in formula  \eqref{Pettis1} (with $f_n=f$ for each $n \in \mathbb{N}$). 
Moreover let  $\delta= \delta(\sigma) >0$ be  that of the  uniform absolute continuity
 of $(m_n)_n$ with respect to $m$, as in formula  (\ref{auc2}).

 Since $m$ is Radon, there is a compact set $K$ such that $m(\Omega \setminus K)< \delta$.
By Urysohn's Lemma there exists a continuous function  $h:\Omega\to [0,1]$
with compact support such that $h(\omega)=1$ for $\omega\in{K}$. Let $g:=\vert f\vert \cdot{h}$.
 Then $g\in{C_0(\Omega)}$.
Hence
\[
\int_{\Omega}\vert f\vert \,dm_n  \leq \int_{K}\vert f\vert \,dm_n+\int_{\Omega \setminus K} \vert f\vert \,  dm_n
 \leq \int_{\Omega}g\,dm_n+\varepsilon.
\]
Since     $(m_n)_n$ converges vaguely to $m$,  then
\[\int_{\Omega}g\,dm_n\longrightarrow \int_{\Omega}g\,dm<+\infty.\]
 Hence
\[
\sup_n\int_{K}\vert f\vert \,dm_n\leq\sup_n \int_{\Omega}g\,dm_n<+\infty.
\]
\end{proof}
\begin{proposition}\label{p2}
Let  $(m_n)_n$  be a sequence in  $\mathcal{M}(\Omega)$   which is uniformly absolutely continuous
with respect to a Radon measure  $m \in \mathcal{M}(\Omega)$   and vaguely  convergent to  $m$.
Moreover let   $f \in C(\Omega)$ be a function   which
 has uniformly  absolutely continuous $(m_n)$-integrals on $\varOmega$.
Then  $f \in L^1(m)$ and
\begin{equation}
	 \lim_n \int_{\Omega} f\,dm_n =  \int_{\Omega} f\,dm.
	\label{convf}
\end{equation}
\end{proposition}
\begin{proof}
 By Proposition \ref{p1}   $\sup_n \int_{\Omega} \vert f\vert  dm_n < +\infty\,$.
 We denote by $(g_k)_k$
 an increasing sequence  of functions   in $C_b(\Omega)$ such that  $0\leq{g_k}\uparrow \vert f\vert $, $m$ a.e.\\
By Proposition \ref{pw} $(m_n)_n$ is also weakly convergent to $m$. Now fix $k \in \mathbb{N}$.
 Let $N_1(k,1)$ be  such that if $n>N_1$
\begin{equation}
	\int_{\Omega} g_k \, dm -1 < \int_{\Omega} g_k \, dm_n .  \label{e112}
\end{equation}
By Proposition \ref{p1} we infer
\begin{equation}
	\int_{\Omega}g_k dm - 1 <   \int_{\Omega} g_k dm_n \leq  \sup_n \int_{\Omega} \vert f\vert  dm_n < \infty\,. \label{e3}
\end{equation}
So, by the Monotone Convergence Theorem applied to  the sequence  $(g_k)_k$ we obtain $f \in L^1(m)$.\\
We are showing now that (\ref{convf}) holds.
We fix $\sigma > 0$.
Since $f \in L^1(m)$ there exists  a positive  $\delta_0 $ such that
for every $A \in \mathcal{B}$ with $m(A) < \delta_0$ then
\begin{equation}
	\int_A \vert f\vert  dm < \sigma. \label{xm}
\end{equation}
Moreover let $\varepsilon(\sigma)>0$ be that of  the
 uniform absolutely continuous $(m_n)$-integrability of $f$  in $\vO$  (with $f_n=f$ for each $n \in \mathbb{N}$)
and  $\delta= \delta(\varepsilon) \in ] 0, \min \{ \varepsilon, \, \delta_0\}[$ be
 that of the absolute continuity of $(m_n)_n$ with respect to $m$.\\
So, if $m(A) < \delta$ then  $\sup_n m_n (A)<\ve$  and
\begin{equation}
 \sup_n \int_{A} \vert f\vert  dm_n  < \sigma. \label{xmn}
\end{equation}

By Urysohn's Lemma one can find a compact set $K$  with $m(K^c) < \delta$
and a function
$h:\Omega\to[0,1]$
in $C_c(\Omega)$
and equal to $1$ on $K$. So $g:=f\cdot{h} \in C_c(\Omega)$. \\
Since the sequence $(m_n)_n$ is vaguely convergent to $m$,    there is $N_2(\sigma) > N_1$ such that for $n>N_2$
\begin{equation} \label{e113}
\left\vert  \int_{\Omega} g \, dm_n - \int_{\Omega} g \, dm \right\vert  < \sigma
\end{equation}

Then by  \eqref{xmn}, \eqref{e113} and \eqref{xm},  for  $n>N_2 $, we have
\begin{align}
&	\left\vert \int_{\Omega} f\,dm_n - \int_{\Omega} f\,dm\right\vert  \leq \nonumber \\
& \left\vert \int_{\Omega} (f-g)\,dm_n \right\vert +
\left\vert \int_{\Omega} g\,dm_n - \int_{\Omega} g\,dm\right\vert +
\left\vert \int_{\Omega} (g-f)\,dm \right\vert  \leq  \nonumber \\
& \left\vert \int_{\Omega} f(1-h)\,dm_n \right\vert +
\left\vert \int_{\Omega} g\,dm_n - \int_{\Omega} g\,dm\right\vert +
\left\vert \int_{\Omega} f(1-h)\,dm \right\vert  \leq  \nonumber \\
& \int_{K^c} \vert f\vert  \,dm_n +
\sigma+
 \int_{K^c} \vert f\vert  \,dm  < 3 \sigma \nonumber
  \end{align}
  and the thesis follows.
\end{proof}


Now our aim is to obtain a limit result
\begin{equation}
\lim_n \int_{A} f_ndm_n =  \int_{A} fdm, \quad \mbox{ for every} A \in \mathcal{B}. \label{passage}
\end{equation}
For the scalar case,
using a Portmanteau's characterization of the vague convergence in metric spaces (see for example \cite{Kall}),  
sufficient conditions  when  $A=\Omega$,  are given
\begin{itemize}
\item[-] 
in  locally compact second countable and  Hausdorff spaces,
 (\cite[Theorems 3.3 and 3.5]{serfozo}), by Serfozo, for the vague and weak convergence respectively, 
 when the sequence $(f_n)_n$ converges continuously to $f$.
Under a domination condition  in the first result
  while, in the second, the uniform  $(m_n)$-integrability of the sequence $(f_n)_n$, with $f_n \geq 0$ for every $n \in \mathbb{N}$,
 is required;
\item[-]  in  locally compact separable metric spaces  (\cite{lerma})
by Herna\-ndez-Lerma and  Lasserre, obtaining a Fatou result and asking for the convergence of the sequence of measures an
 inequality of the $\liminf$ of the $m_n$ on each Borelian set;
\item[-] in  metric spaces  (\cite{Fein-arX1807,Fein-arX1902}), where  the authors obtained a dominated convergence
   result for sequences of equicontinuous functions $(f_n)_n$ satisfying the uniform $(m_n)$-integrability.\\
\end{itemize}

In Theorem \ref{Th1-4},
taking into account Remarks  \ref{RWV} and \ref{2.6jmaa},
we extend   \cite[Theorem  3.5]{serfozo},
 obtaining a sufficient condition when the convergence is vague, the functions $f_n$ are real valued and using the
 uniformly absolutely   continuous  $(m_n)$-integrability of the sequence $(f_n)_n$.
 Later, in Section \ref{sec4},  we will also extend it to the vector and multivalued cases making use of the Pettis integrability.\\

We assume only that $\Omega$ is a locally compact Hausdorff space and then,
in our setting, $\Omega$ is  a Tychonoff space, i.e. a completely regular Hausdorff space (\cite[Theorem 3.3.1]{engel}).
 So we are able to 
use the following  Portmanteau's characterization of the vague convergence for positive measures given  in \cite{Bog}.\\

\begin{theorem}{\rm (\cite[Corollary 8.1.8 and Remark 8.1.11]{Bog})}\label{Port}
 Let $\Omega $ be an arbitrary completely regular space and let  $m$ and $m_n$, $n \in\N$,
	be  measures in   $\mathcal{M}(\Omega)$ with $m$ Radon and assume that $\lim_{n} m_n(\Omega)=m(\Omega)$.
	Then the following are equivalent:
	\begin{description}
	\item[\ref{Port}.i)] $(m_n)_n$ is  vaguely   convergent to $m$;
	\item[\ref{Port}.ii)] for any closed set $F \subset \Omega$, $\limsup _n m_n(F) \leq m(F)$.
	\end{description}
	\end{theorem}

So we have
	
\begin{theorem}\label{Th1-4}

 Let $m$ and $m_n$, $n \in\N$,	be
 measures in   $\mathcal{M}(\Omega)$,  with $m$    Radon.  Let $f, f_n \in \mathcal{F}(\Omega)$.
	Suppose that
	\begin{description} 				
		\item[(\ref{Th1-4}.i)]  $f_n(t)  \rightarrow f(t)$,  $m$-a.e.;
		\item[(\ref{Th1-4}.ii)]    $f \in C(\Omega)$;
	\item[(\ref{Th1-4}.iii)]
$(f_n)_n$ and $f$  have uniformly absolutely   continuous  $(m_n)$-integrals  on $\varOmega$; 	
		\item[(\ref{Th1-4}.iv)] $(m_n)_n$ is  vaguely   convergent to $m$ and
uniformly  absolutely continuous with respect to $m$.
	\end{description} 	
	Then  $f \in L^1(m)$ and
	\begin{equation}
		 \lim_n \int_{A} f_ndm_n =  \int_{A} fdm \quad \mbox{ for every  } A \in \mathcal{B}. \label{conv4-4}
	\end{equation}
	
\end{theorem}

\begin{proof}	
	By Proposition \ref{p2} the function $f \in L^1(m)$.
We proceed by steps.
\begin{description}
\item[Step 1]  We prove \eqref{conv4-4} for $A = \Omega$.\\
 Fix $\varepsilon >0$ and let
	$\delta:= \min\left\{\frac{\varepsilon}{6},   \delta(\frac{\varepsilon}{6}) , \delta_f(\frac{\varepsilon}{6}) \right\} >0$
	where
	$\delta_f(\frac{\varepsilon}{6})$ is that of the absolute continuity of $\int_{.} \vert f\vert  dm$ ,  and
by (\ref{Th1-4}.iii)
	$\delta(\frac{\varepsilon}{6})$ is  that of  (\ref{Pettis1})  for both $(f_n)_n$ and $f$ with respect to $(m_n)_n$.\\	
By the hypothesis (\ref{Th1-4}.iv)    let 	$0<\delta_0 <\delta$ be such that		
	\begin{equation}	
\left(E \in  \mathcal{B} \quad \mbox{and} \ \  m(E)<\delta_0 \right)
	 \,\, \Longrightarrow \,\, \sup_n  m_n (E) < \delta. \label{auc4-qui}
	\end{equation}
	By   the Egoroff's Theorem, we can find a compact set $K$ such that $f_n \rightarrow f$ uniformly on $K$
	and $m(K^c) <  \delta_0$.\\
We observe that by condition \ref{Th1-4}.iv) and by Proposition \ref{pw} $(m_n)_n$ weakly converges to $m$ and then $\lim_nm_n(\Omega)=m(\Omega)$.  So by  Theorem \ref{Port}, let   $N_0 \in\N$ be such that
\begin{equation}
m_n(K) < m(K) + 1 , \label{Pettis200}
\end{equation}
for every $n>N_0$.
Moreover, since the convergence is uniform on $K$, let    $N_1 > N_0 \in\N$ be such that

\begin{equation}
	\vert f_n(t)-f(t)\vert  <  \frac{ \varepsilon }{6(m(K)+1)}  , \label{800}
\end{equation}
for every $t \in K$ and $n>N_1$. Then, for all $n>N_1$,

\begin{equation}
	\int_{K}\vert f_n-f\vert dm < \frac{ \varepsilon }{6}. \label{1100}
\end{equation}

Therefore by (\ref{Pettis200}) and (\ref{800}) we obtain, for every for $n> N_1$,
\begin{equation}
	\int_{K}\vert f_n-f\vert dm_n \leq \frac{ \varepsilon }{6(m(K)+1)} \cdot m_n(K) <
\frac{ \varepsilon }{6}. \label{900}
	 \end{equation}
Since  $m(K^c) < \delta_0 $ by (\ref{auc4-qui}) it follows that  $ m_n (K^c) < \delta $ for every 
$n \in \mathbb{N}$.
Moreover,  by hypothesis \ref{Th1-4}.iii) and by the choice of $\delta$,  we have that

	\begin{equation}
	\max \, \left\{ \int_{K^c} \vert f\vert dm, \,\, \int_{K^c} \vert f_n\vert dm_n, \,\,	\int_{K^c} \vert f\vert dm_n  \right\} < \frac{ \varepsilon}{6}.
	\label{per-f-risp-mn}
	\end{equation}
	
By  Urysohn's Lemma let $h:\Omega\to[0,1]$ be a continuous function with compact support 
 equal to $1$ on $K$.
 Then $g:=f\cdot{h} \in C_c(\Omega)$ and
	by  (\ref{per-f-risp-mn}) 	we  have

	\begin{equation}
	\max \, \left\{ 	\int_{K^c}\vert f-g\vert  dm_n, \,\, \int_{K^c}\vert f-g\vert  dm \right\} < \frac{ \varepsilon}{6}. \label{l200}
	\end{equation}
	
	Moreover, since $(m_n)_n$ is vaguely  convergent to $m$, let $n > N_2 \geq N_1$ be such that
	
	\begin{equation}
		\Big\vert  \int_{\Omega} g \ dm- \int_{\Omega} g \ dm_n \Big\vert  < \frac{ \varepsilon}{6}. \label{Pettis10}
	\end{equation}
	
	Therefore by  (\ref{900})--(\ref{Pettis10})
	for $n > N_2$ we obtain
	\begin{align}\label{eq:3.5.b}
		& \nonumber
\left\vert  \int_{\Omega} f dm -   \int_{\Omega} f_n dm_n \right\vert  \leq
		\left\vert  \int_{\Omega} (f_n -f) dm_n \right\vert  +   \left\vert  \int_{\Omega} f dm - 
		\int_{\Omega} f dm_n  \right\vert  \\
		&\leq \nonumber
 \left\vert  \int_{K} (f_n -f) dm_n \right\vert  +  \int_{K^c} \vert f_n\vert  dm_n + \int_{K^c} \vert f\vert  dm_n +
		 \left\vert  \int_{\Omega} f dm - \int_{\Omega} f dm_n  \right\vert \\
		&\leq   \frac{ \varepsilon}{2}+\left\vert  \int_{\Omega} f dm - \int_{\Omega} f dm_n  \right\vert \\
		&\leq  \nonumber
		\frac{ \varepsilon}{2} +  \int_{K^c} \vert f - g\vert dm +\int_{K^c} \vert f - g\vert dm_n +
		\left\vert  \int_{\Omega} g \ dm- \int_{\Omega} g \ dm_n \right\vert 
		<  \varepsilon
	\end{align}
	so (\ref{conv4-4}) follows for $A=\Omega$.
\item[Step 2] 
Now we are proving that
\eqref{conv4-4} is valid for an arbitrary compact set $K$.\\
Let once again, $\varepsilon> 0$ be fixed. 
By (\ref{Th1-4}.iii) and (\ref{Th1-4}.iv)
 there exist
$\delta_1,\delta_2>0$   such that:
\begin{description}
\item[$j_1$)] if  $m_n(E)<\delta_2$, then $\displaystyle{\int_E \vert f_n\vert \,dm_n<\ve}$ for every $n \in \mathbb{N}$;
\item[$j_2$)] if  $m(E)<\delta_1$, then  $m_n(E)<\delta_2 $ for every $n \in \mathbb{N}$;
\item[$j_3$)]  if  $m(E)<\delta_1$, then $\displaystyle{\int_E \vert f\vert \,dm<\ve}$.
\end{description}
Let now 
 $U\supset{K}$  be an open set such that $m(U\smi{K})<\delta_1$.
 Then let $g:\vO\to[0,1]$  be continuous and such that $g=1$ on $K$ and zero on $U^c$.\\
Observe that the sequence $(f_n g)_n$ and  the function $fg$ satisfy all the hypotheses of Theorem \ref{Th1-4} so, for the Step 1,  we have
\[
	 \lim_n \int_{U} f_n g\, dm_n	=  \lim_n \int_{\Omega} f_n g\, dm_n =  \int_{\Omega} fg \,  dm =
\int_{U} fg \,  dm .
\]
Then, by the previous inequalities and  for  $n$ sufficiently large,  we have
\begin{align}
& \Big\vert  \int_K f_n\,dm_n-\int_K f\,dm \Big\vert  =
\Big\vert  \int_K f_n g \,dm_n - \int_K fg \,dm 
+ \big( \int_{K^c} f_n g \,dm_n- \int_{K^c} f g\,dm \big)   + 
\nonumber \\ 
&  - \big( \int_{K^c} f_n g\, dm_n - \int_{K^c} f g \,dm \big) 
\Big \vert  
\nonumber \\
&\leq \left\vert \int_{\vO}f_ng\,dm_n -\int_{\vO}fg\,dm\right\vert  + 
\left\vert \int_{K^c}f_ng\,dm_n-\int_{K^c}fg\,dm\right\vert  \nonumber \\
&= \left\vert \int_{\vO}f_ng\,dm_n-\int_{\vO}fg\,dm\right\vert 
+ \left\vert \int_{U\smi{K}}f_ng\,dm_n-\int_{U\smi{K}}fg\,dm\right\vert  \nonumber  \\
&< \left\vert \int_{\vO}f_ng\,dm_n-\int_{\vO}fg\,dm\right\vert  + \int_{U\smi{K}}\vert f_n\vert \,dm_n +
 \int_{U\smi{K}}\vert f\vert \,dm <3\ve. \nonumber
\end{align}
\item[Step 3]
Let now $B$ a Borelian set and let $\varepsilon> 0$,
$\delta_1,\delta_2>0$  as in Step 2.
Let $C_1$ be a compact set with $C_1 \subset B$
such that
$m(B \setminus C_1) <\delta_1$.
So
\begin{align}
&\left\vert \int_B f_n\,dm_n-\int_B f\,dm\right\vert 
\leq \left\vert \int_{C_1 }f_n \,dm_n -\int_{C_1} f \,dm\right\vert 
+ \nonumber \\
 & + \left\vert \int_{B \setminus C_1}f_n \,dm_n-\int_{B \setminus C_1}f \,dm\right\vert   \nonumber  \\
& \leq \left\vert \int_{C_1}f_n \,dm_n-\int_{C_1}f \,dm\right\vert  + \int_{B\smi{C_1}}\vert f_n\vert \,dm_n + \int_{B\smi{C_1}}\vert f\vert \,dm. \nonumber
\end{align}
So the assertion follows from $j_1)-j_3$) and the compact case in Step 2 and
 this ends the proof.
\end{description}
\end{proof}

\begin{corollary}
Let $m$ and $m_n$, $n \in\N$,	be
 measures in   $\mathcal{M}(\Omega)$,  with $m$    Radon.  	
		If $(m_n)_n$ is  vaguely   convergent to $m$ and
uniformly  absolutely continuous with respect to $m$, then  $(m_n)_n$ converges setwisely to $m$.
\end{corollary}
\begin{proof} It is a consequence of Theorem \ref{Th1-4} if we assume $f_n=f\equiv 1$ for every $n \in \mathbb{N}$.
\end{proof}

 \begin{remark}\label{positivepart}
		\rm
		\phantom{a}
		
\begin{description}
\item[\ref{positivepart}.a)] We observe  that under the hypotheses of  Theorem \ref{Th1-4}, if  $f \in C(\Omega)$,
	 then also  $f^{\pm}$  are in $C(\Omega)$ and $f_n^{\pm}(t)  \rightarrow f^{\pm}(t)$   $m$-a.e. as $n 			\rightarrow \infty$.
	In fact
	\begin{align}
	& \Big\vert  \, \vert f_n\vert  - \vert f\vert \, \Big\vert  \leq \big\vert  f_n - f \big\vert  \nonumber
	\\
	& 2 f_n^+ = f_n + \vert f_n \vert  \to f + \vert f \vert  = 2f^+; \nonumber \\
	& 2 f_n^- =  \vert f_n\vert  - f_n \to  \vert f\vert  - f = 2f^-. \nonumber
	\end{align}

	 Moreover also $(f_n^{\pm})_n$ and $f^{\pm}$  satisfy condition \ref{Th1-4}.iii) since
		\[
	 f_n^{\pm} \leq \vert f_n\vert  \quad \mbox{and} \quad
	f^{\pm} \leq \vert f\vert .
		\]
 	Therefore  in the hypotheses of Theorem \ref{Th1-4}
	we get also
	\[
	\lim_n \int_{A} f^{\pm}_ndm_n =  \int_{A} f^{\pm}dm, \quad
	\mbox{for every } A \in \mathcal{B}.
	\]	
\item[\ref{positivepart}.b)]
Theorem \ref{Th1-4} is still valid if we replace condition \ref{Th1-4}.i) with
	\begin{description}
		\item[ \ref{Th1-4}.i')] $f_n$ converges in $m$-measure to $f$.
	\end{description}

	In fact, by \ref{Th1-4}.i'),
	 there exists a subsequence of $(f_{n_k})_k$ which converges $m$-a.e. to $f$.
	Then  Theorem \ref{Th1-4} is true for such subsequence. So
	this implies that the result of this theorem, equality (\ref{conv4-4}),  is still  valid for the initial sequence
	(with convergence in $m$-measure) because if, absurdly, a subsequence
	 existed in which it is not valid, there would be a contradiction.
\end{description}
\end{remark}

A simple consequence of the  Theorem \ref{Th1-4} is the following
\begin{theorem}\label{Th1-ter}
Let  $m$ and $m_n$, $n \in\N$,	be measures in   $\mathcal{M}(\Omega)$,  with $m$  
Radon. Let $f, f_n  \in \mathcal{F}(\Omega)$.
	Suppose that
	\begin{description} 					
		\item[(\ref{Th1-ter}.i)]  $f_n(t)  \rightarrow f(t)$,  $m$-a.e.;
		\item[(\ref{Th1-ter}.ii)]  $f \in C_b(\Omega)$;
		\item[(\ref{Th1-ter}.iii)]  $(f_n)_n$ has uniformly absolutely continuous  $(m_n)$-integrals
			on $\varOmega$;
		\item[(\ref{Th1-ter}.iv)]  $(m_n)_n$ is vaguely  convergent to $m$ and uniformly
		absolutely continuous with respect to $m$.
		\end{description} 	
	Then  	
		\[
			\lim_n \int_{A} f_ndm_n =  \int_{A} fdm, \quad \mbox{for every } A \in \mathcal{B}.
		\]	
\end{theorem}
\begin{proof}
The assertion follows from Theorem \ref{Th1-4}  since  $f$
has uniformly absolutely continuous  $(m_n)$-integrals, in fact it is enough to take the pair
$(\varepsilon, \delta(\varepsilon/M))$ where $M > \sup_{t \in \Omega} \vert f(t)\vert $.
\end{proof}

\section{The multivalued  and the vector cases}\label{sec4}

\subsection{The multivalued case}

Let $X$ be a Banach space with   dual $X^*$ and let $B_{X^*}$ be the unit ball of $X^*$.
The symbol  $cwk(X)$ denotes
the family of all weakly compact  and convex subsets of $X$.  For every $C \in cwk(X)$ the
{\it  support function of}   $\, C$ is denoted by $s( \cdot, C)$ and
defined on $X^*$ by $s(x^*, C) = \sup \{ \langle x^*,x \rangle \colon  \  x \in C\}$.
Recall that $X$ is said to be {\it weakly compact generated }
 (briefly WCG) if it possesses a weakly compact subset
 $K$ whose linear span is dense in $X$. \\
 A map $\Gamma: \Omega \to cwk(X)$ is called a {\it multifunction}.
A space $Y \subset X$ {\it $m$-determines}
a multifunction $\Gamma$
 if $s(x^*, \Gamma)=0$ $m$ a.e. for every $x^* \in Y^{\perp}$, where the exceptional sets depend on $x^*$. \\
A multifunction $\Gamma$ is said to be
\begin{itemize}
	\item {\it scalarly measurable} if
	$t\to s(x^*,\Gamma(t))$ is measurable, for every $x^*\in X^*$;
\item  {\it scalarly $m$-integrable} if
	$t\to s(x^*,\Gamma(t))$ is $m$-integrable, for every $x^*\in X^*$, where $m \in \mathcal{M}(\Omega)$;
\item  {\it scalarly continuous} if for every $x^*\in X^*$, $t\to s(x^*,\Gamma(t))$ is continuous.
\end{itemize}
A multifunction  $\Gamma: \Omega \to cwk(X)$ is said to be  {\it Pettis integrable }   in $cwk(X)$
  with respect to  a measure  $m$ (or shortly Pettis $m$-integrable) if
 $\Gamma$ is scalarly $m$-integrable  and for every measurable set
 $A$, there exists $M_{\Gamma}(A) \in cwk(X)$ such that
\[
s(x^*,M_{\Gamma}(A))=\int_A s\big(x^*,\Gamma \big)\, dm \qquad \mbox{\rm for  all }   \;     x^*\in X^*.
\]

We set $\dint_A \Gamma dm := M_{\Gamma}(A)$.\\ For the properties of Pettis $m$-integrability in the
 multivalued case  we refer to \cite{can1,can2,can3,xcoletti,M2,mu8}, while for the vector case we refer to 
 \cite{M}.
If $\Gamma$ is single-valued we obtain the classical definition of Pettis integral for vector function.
\\

 Given a sequence of multifunctions
we introduce now some definitions of uniformly absolutely   continuous scalar integrability using Definition \ref{a-c}.\\

\begin{definition}\label{def3.1}
\rm
For every $n \in\N$, let $m_n$ be a measure in $\mathcal{M}(\Omega)$
and  let $\Gamma_n :\vO \rightarrow cwk(X)$ be a  multifunction which is  scalarly $m_n$-integrable.
We say that the sequence $(\Gamma_n)_n$
 has {\it uniformly absolutely   continuous scalar $(m_n)$-integrals   on
$\vO$}  if for every $ \varepsilon  >0$
there exists $\delta >0$ such that, for every $n\in\N$ and $A \in  \mathcal{B}$, it is
\begin{equation}
 m_n(A)<\delta  \ \Rightarrow \
	\sup\left\{\int_A\vert s(x^*, \Gamma_n)\vert dm_n \colon \parallel x^*\parallel  \leq 1\right\}< \varepsilon.\, \label{formula-def3.1}
\end{equation}
Analogously a multifunction   $\Gamma$
has {\it  uniformly absolutely   continuous scalar $(m_n)$-integrals   on $\varOmega$}
 if previous condition (\ref{formula-def3.1}) holds for  $\Gamma_n:=\Gamma$ for every $n \in \mathbb{N}$.	
Moreover we say that $\Gamma$ has
{\it  uniformly absolutely   continuous scalar $m$-integrals   on $\varOmega$}
 if, in formula (\ref{formula-def3.1}), it is $\Gamma_n:=\Gamma$ and
 $m_n=m$ for every $n \in \mathbb{N}$. In this case we have, for every $A \in  \mathcal{B}$,
\begin{equation}
\hskip-.4cm	
 m(A)<\delta  \ \Rightarrow \
	\sup\left\{\int_A\vert s(x^*, \Gamma)\vert dm \colon \parallel x^*\parallel \leq 1\right\}< \varepsilon. \label{formula-def4.1}
\end{equation}
\end{definition}

\begin{theorem}\label{Thmulti2} Let  $\Gamma, \Gamma_n$, $n \in\N$,  be scalarly measurable multifunctions.
 Moreover let $m, \  m_n$, $n \in\N$,  be  measures in  ${\mathcal M}(\Omega)$ and let $m$ be  Radon.
		Suppose that	
\begin{description}	
\item[(\ref{Thmulti2}.j)]  $ (\Gamma_n )_n$ and $\Gamma$
		 have   uniformly absolutely continuous  scalar $(m_n)$-integrals on $\varOmega$;
\item[(\ref{Thmulti2}.jj))]  $s(x^*, \Gamma_n)  \rightarrow s(x^*, \Gamma)$  $m$-a.e. 
	for each $x^* \in X^*$;
 \item[(\ref{Thmulti2}.jjj)]  $\Gamma$ is scalar continuous; 			
\item[(\ref{Thmulti2}.jv)]   $(m_n)_n$ is vaguely   convergent to $m$ and  uniformly absolutely continuous with respect to $m$;	
\item[(\ref{Thmulti2}.v)]	each  multifunction $\Gamma_n$ is Pettis $m_n$-integrable.	
\end{description}
	Then the multifunction $\Gamma$ is Pettis $m$-integrable  in $cwk(X)$ and
\begin{equation}
	 \lim_n s\Big( x^*, \int_{A} \Gamma_n\, dm_n \Big)=  s\Big( x^*, \int_{A} \Gamma\, dm \Big), \label{fuori}
\end{equation}
	for every $x^* \in X^*$ and for every $A \in \mathcal{B}$.
\end{theorem}

\begin{proof}
Let $x^* \in X^*$ be fixed.
 Then the sequence of functions $\big(s(x^*, \Gamma_n)\big)_n$ and the function $s(x^*, \Gamma)$
 defined on $\Omega$ satisfy  the assumptions  of  Theorem \ref{Th1-4}.
So,  for each $A \in \mathcal{B}$
\begin{equation}
	\lim_n \int_{A} s(x^*, \Gamma_n)\, dm_n =  \int_{A}s (x^*, \Gamma)\, dm. \label{eccola!}
\end{equation}
In order to  prove that $\Gamma$ is Pettis $m$-integrable,
following \cite[Theorem 2.5]{M2}, it is enough to show that
 the sublinear operator
$T_{\Gamma}: X^* \to L^1(m)$, defined as $T_{\Gamma}(x^*)=s(x^*, \Gamma)$
is weakly compact (step  $C_w$) and that $\Gamma$ is determined by a $WCG$ space $Y \subset X$ (step $D$).
\begin{description}
\item[$C_w$)] First of all we prove that the operator $T_{\Gamma}$ is bounded.
By (\ref{Thmulti2}.jjj) $\Gamma$ is scalar $m$-integrable.
Therefore
$\Gamma$ is Dunford-integrable
in $cw^*k(X^{**})$,  where  $X^{**}$ is endowed with the $w^*$-topology, and
for every $A \in \mathcal{B}$ let $M^D_{\Gamma} (A) \in cw^*k(X^{**})$ be such that
\begin{equation} \label{D1}
	s(x^*, M^D_{\Gamma} (A)) =\int_As(x^*,\Gamma) dm < +\infty,
\end{equation}
for every $x^* \in X^*$.
So $s(x^*,M^D_{\Gamma} ( \cdot))$ is a scalar measure and

\[\int_{\Omega} \vert s(x^*,\Gamma) \vert dm \leq
2 \sup_{A\in \mathcal{B} }\left\vert  \int_A s(x^*,\Gamma) dm \right\vert 
 < +\infty.\]
Hence, the set 
\mbox{$\displaystyle{\bigcup_{A \in \mathcal{B}} }M_{\Gamma}^D(A) \subset X^{**}$}
  is bounded, by  the Banach– Steinhaus Theorem, and
\[
\sup_{\parallel x^*\parallel  \leq 1} \int_{\Omega} \vert s(x^*,\Gamma) \vert dm \leq 2 \sup \left\{ \parallel  x \parallel : x \in
\bigcup_{A \in \mathcal{B}} M_{\Gamma}^D(A) \right\} <+ \infty.\]
Since  the set $\big\{s(x^*,\Gamma) : \parallel x^*\parallel \leq1 \big\}$ is bounded in $L^1(m)$,
 the operator $T_{\Gamma}$ is bounded.

In order to obtain the weak compactness of the operator $T_{\Gamma}$ it is enough to prove that
$\Gamma$ has absolutely   continuous scalar $m$-integrals   on $\varOmega$.
Let  $x^* \in B_{X^*}$ be fixed. Now fix $\varepsilon >0$ and let $\sigma(\varepsilon) >0$
 satisfy  (\ref{Thmulti2}.j).
Moreover  let $\delta(\sigma) >0$  verify 
 { (\ref{Thmulti2}.jv)}.\\
Let $E \in  \mathcal{B}$ be such that $m(E) < \delta$ and  set
\[
 E^+=\{ t \in E : s(x^*, \Gamma(t)) \geq 0\} \quad
 E^-=\{ t \in E : s(x^*, \Gamma(t)) <0\}.
\]
By \eqref{eccola!} let now  $ N_{x^*} \in \mathbb{N}$ be an
 integer  such that for every $n \geq N_{x^*}$
\[ \left\vert \int_{E^{\pm}} s(x^*, \Gamma)dm \right\vert  < 
\left\vert  \int_{E^{\pm}} s(x^*, \Gamma_{n})dm_{n} \right\vert  + \frac{\varepsilon}{2}.\]
So, for every $n \geq N_{x^*}$,
\begin{align}
\int_E \vert s(x^*, \Gamma)\vert  dm &= \int_{E^+} s(x^*, \Gamma)dm + \left\vert \int_{E^-} s(x^*, \Gamma)dm \right\vert  \nonumber\\
&< \left\vert  \int_{E^+} s(x^*, \Gamma_{n})dm_{n} \right\vert  +
 \left\vert  \int_{E^-} s(x^*, \Gamma_{n})dm_{n} \right\vert  +\varepsilon. \nonumber
\end{align}
Since, by  (\ref{Thmulti2}.jv), it is in particular  $m_{n}(E) < \sigma$ for every $n \geq N_{x^*}$,   we get
\[
\int_E \vert s(x^*, \Gamma)\vert  dm
\leq  \int_E \vert s(x^*, \Gamma_{n})\vert  dm_{n} + \varepsilon < 2 \varepsilon
\]
so $\Gamma$ has   uniformly absolutely   continuous scalar $m$-integral   on $\varOmega$.\\
\item[$D$)]
We have to show the existence
of a  $WCG$ subspace of $X$ which determines
 $\Gamma$.
Since  $\Gamma_n$ is Pettis $m_n$-integrable, for every $n \in\N$,
let $Y_n \subseteq X$ be a $WCG$ space  generated by a
 set $W_n \in cwk( B_{X^*})$
which  $m_n$-determines $\Gamma_n$,
by \cite[Theorem 2.5]{M2}.\\
We may suppose, withous loss of generality that each $W_n$ is absolutely convex, by
 Krein-Smulian's Theorem.
Let $Y$ be the WCG space generated by  $W: =\sum 2^{-n} W_n$.
We want to prove that  $\Gamma$ is $m$-determined by $Y$.\\
If $y^*\in{Y^\perp}$, then $y^*\in{Y_n^\perp}$ for each $n$, hence $s(y^*,\Gamma_n)=0$ $m_n$-a.e.
Applying (\ref{eccola!}) with $A=\Omega^+:=\{t  : s(y^*, \Gamma(t)) \geq 0\}$
($A = \Omega^-:=\{t  : s(y^*, \Gamma(t)) < 0\}$)
we get
\[ \int_{\Omega^{\pm}}s (y^*, \Gamma)\, dm=\lim_n \int_{\Omega^{\pm}} s(y^*, \Gamma_n)\, dm_n  =0.\]
Therefore $s(y^*,\Gamma(t))=0$ $m$-a.e. on the set $\Omega$.
Thus, $Y$ $m$-determines the multifunction $\Gamma$ and the Pettis $m$-integrability of $\Gamma$  follows.
Moreover \eqref{fuori} follows from \eqref{eccola!}.
\end{description}
\end{proof}
As an immediate consequence of the previous theorem we have a result for the vector case:
\begin{corollary}\label{Thmulti2c}
Let  $g, g_n :\Omega \rightarrow  X$, $n \in\N$,  be scalarly measurable functions.
 Moreover let $m, \  m_n$, $n \in\N$,  be  measures in  ${\mathcal M}(\Omega)$ and let $m$ be  Radon.
		Suppose that	
\begin{description}	
\item[(\ref{Thmulti2c}.j)]  $ (g_n )_n$ and $g$
		 have   scalarly uniformly absolutely continuous $(m_n)$-integrals on $\varOmega$;
\item[(\ref{Thmulti2c}.jj))]  $g_n  \rightarrow g$ scalarly  $m$-a.e.
	 where the null set depends on $x^* \in X^*$;
 \item[(\ref{Thmulti2c}.jjj)]  $g$ is scalar continuous; 			
\item[(\ref{Thmulti2c}.jv)]   $(m_n)_n$ is vaguely   convergent to $m$ and uniformly
		absolutely continuous with respect to $m$;	
\item[(\ref{Thmulti2c}.v)]	each   $g_n$ is Pettis $m_n$-integrable.	
\end{description}
	Then $g$ is Pettis $m$-integrable  in $X$ and	
\[
 \lim_n x^*\left(\int_{A}  g_n\, dm_n \right)=
 x^* \left(\int_{A}  g\, dm \right),
\]
	for every $x^* \in X^*$ and $A \in \mathcal{B}$.
\end{corollary}

We conclude with the following result that  holds in a general measure space without any topology on the space $\Omega$.

 \begin{proposition}\label{Thmulti2bis} Let $\Omega $ be a measurable space on a $\sigma$-algebra $\mathcal{A}$
  and let $\Gamma, \Gamma_n$, $n \in\N$,  be scalarly measurable multifunctions.
 Moreover let $m, \  m_n$, $n \in\N$,  be  measures in  ${\mathcal M}(\Omega)$.
		Suppose that	
	\begin{description}	
	\item[(\ref{Thmulti2bis}.j)]  $ (\Gamma_n )_n$
		 have   scalarly uniformly absolutely continuous $(m_n)$-integrals on $\varOmega$;
	 \item[(\ref{Thmulti2bis}.jj)]  $\Gamma$ is  scalarly $m$-integrable; 			
	\item[(\ref{Thmulti2bis}.jjj)]  $(m_n)_n$ is uniformly
		absolutely continuous with respect to $m$;
	\item[(\ref{Thmulti2bis}.jv)] each  multifunction $\Gamma_n$	is Pettis $m_n$-integrable.
\item[(\ref{Thmulti2bis}.v)]   for every $A \in \mathcal{A}$ and for every $x^* \in X^*$ it is
\[
 \lim_n  \int_{A}  s(x^*, \Gamma_n) \, dm_n =  \int_A s(x^*,   \Gamma) \, dm.
\]

	\end{description}
	Then the multifunction $\Gamma$ is Pettis $m$-integrable.
	\end{proposition}

\begin{proof}
The weak compactness  of the sublinear operator
$T_{\Gamma}: X^* \to L^1(m)$, defined as $T_{\Gamma}(x^*)=s(x^*, \Gamma)$ can be proved as
Theorem \ref{Thmulti2}, taking into account    hypotheses (\ref{Thmulti2bis}.jj), (\ref{Thmulti2bis}.jjj) and (\ref{Thmulti2bis}.v).\\
Moreover,
the proof that  $\Gamma$ is determined by a $WCG$ space $Y \subset X$ follows as in Theorem  \ref{Thmulti2}, taking into account
 hypotheses (\ref{Thmulti2bis}.j), (\ref{Thmulti2bis}.jjj), (\ref{Thmulti2bis}.jv) and (\ref{Thmulti2bis}.v).
Therefore  $\Gamma$ is Pettis $m$-integrable.\\
\end{proof}

At this point it is worth to observe that a similar result   has been  proved in \cite[Theorem 3.2]{dmms}
 under the hypothesis of the setwise convergence of the measures. Here instead of the setwise convergence we assume
the uniform absolute continuity of $(m_n)_n$ with respect to $m$.\\

For the vector case, as before, we have:
 \begin{corollary}\label{cormulti2bis}  Let $\Omega $ be a measurable space on a $\sigma$-algebra $\mathcal{A}$
  and let $g, g_n :\varOmega \to X$, $n \in\N$,  be scalarly measurable functions.
 Moreover let $m, \  m_n$, $n \in\N$,  be  measures in  ${\mathcal M}(\Omega)$.
		Suppose that	
\begin{description}	
\item[(\ref{cormulti2bis}.j)]  $ (g_n )_n$
		 have   scalarly uniformly absolutely continuous $(m_n)$-integrals on $\varOmega$;
\item[(\ref{cormulti2bis}.jj)]  $g$ is  scalar $m$-integrable; 			
\item[(\ref{cormulti2bis}.jjj)]  $(m_n)_n$ is uniformly
		absolutely continuous with respect to $m$;
\item[(\ref{cormulti2bis}.jv)]	each function $g_n$	is Pettis $m_n$-integrable;
\item[(\ref{cormulti2bis}.v)]   for every $A \in \mathcal{A}$ and for every $x^* \in X^*$ it is
\[
 \lim_n  \int_{A}  x^*g_n \, dm_n =  \int_A x^*g \, dm.
\]
\end{description}
	Then the function $g$ is Pettis $m$-integrable.
	\end{corollary}

\section{Conclusion}\label{sec13}

Some limit theorems  for the sequences
$\left(\int f_n\,dm_n \right)_n$ are presented
for  vector and multivalued Pettis integrable functions when the sequence $(m_n)_n$ vaguely converges to a measure $m$.
 The results are obtained thanks to a limit result obtained for the scalar case (Theorem \ref{Th1-4}).\\

{\small
This research has been accomplished within the UMI Group TAA “Approximation Theory and Applications”, the  G.N.AM.P.A. of INDAM and the
 Universities of Perugia and Palermo. \\ It was  supported by
 Ricerca di Base 2019 dell'Universit\`a degli Studi di Perugia - "Integrazione, Approssimazione, Analisi Nonlineare e loro Applicazioni";
 "Metodiche di Imaging non invasivo mediante angiografia OCT
sequenziale per lo studio delle Retinopatie degenerative dell'Anziano (M.I.R.A.)", funded by FCRP, 2019;
 F.F.R. 2023 - Marraffa dell'Universit\`a degli Studi di Palermo.
}


\Addresses

\end{document}